\documentclass[11pt]{article}
\usepackage[top=2.5cm,bottom=2.5cm,left=2.5cm,right=2.5cm]{geometry}
\usepackage{mathrsfs}
\usepackage{pstricks}
\usepackage{amsmath,amssymb,graphicx,amsthm,tabularx,array}
\usepackage{hyperref}

\theoremstyle{plain}
\newtheorem{theorem}{Theorem}[section]

\newtheorem{coro}[theorem]{Corollary}
\newtheorem{lemma}[theorem]{Lemma}
\newtheorem{defn}[theorem]{Definition}

\theoremstyle{definition}

\newtheorem{other}{}

\title{\bf The  largest spectral radius of  uniform hypertrees with a given size of matching\footnote{This work was partially supported by National Natural
Science Foundation of China (Nos. 11561032, 11571222, 11471210), the Jiangxi Science Fund for Distinguished Young Scholars (No. 20171BCB23032) and the funds of the Education Department
 of Jiangxi Province (No. GJJ150345).  }}

\author{Li Su$^{a,b}$, \, Liying Kang$^{a}$\footnote{Corresponding author. E-mail addresses: lykang@shu.edu.cn\,(L. Kang),  suli@jxnu.edu.cn\,(L. Su), lhh@mail.ustc.edu.cn\,(H. Li), efshan@i.shu.edu.cn\,(E. Shan)}, \, Honghai Li$^{b}$, \, Erfang Shan$^{c}$\\[5mm]
\small $^a$ Department of Mathematics, Shanghai University, Shanghai 200444,  China\\
\small   $^b$College of Mathematics and Information Science, Jiangxi Normal University\\
\small  Nanchang, Jiangxi 330022,  China\\
\small   $^c$School of Management, Shanghai University,
Shanghai 200444,  China}
\date{}
\begin{document}
\maketitle
\begin{abstract}
In this paper, using the theory of matching polynomial of hypertrees and ordering of hypertrees,   we  determine the  largest spectral radius of  hypertrees    with $m$  edges and given size of matching.

\vspace{3mm}


\vspace{2mm}


\end{abstract}

\section{Introduction}

Hypergraphs are systems of sets which are conceived as natural extension of graphs. A {\em hypergraph} $\mathcal{H}=(V(\mathcal{H}), E(\mathcal{H}))$ is a finite set $V(\mathcal{H})$ of elements, called {\em vertices}, together with a finite multiset $E(\mathcal{H})$ of subsets of $V(\mathcal{H})$, called {\em
hyperedges} or simply {\em edges}.   For a vertex $v$ in $\mathcal{H}$, let $E_v(\mathcal{H})$ (or simply $E_v$) represent the set of edges containing $v$, i.e., $E_v(\mathcal{H}) = \{e\in E \mid v\in e\}$, and the {\em degree} of  vertex $v$   is the cardinality $|E_v|$.
A hypergraph $\mathcal{H}$ is \emph{$r$-uniform}  if every edge $e\in E(\mathcal{H})$ contains precisely $r$ vertices.
 A vertex with degree one is a \textit{core vertex}, and a vertex with degree larger than one is an \textit{intersection vertex}.
 If any two edges in $\mathcal{H}$ share at most one vertex, then $\mathcal{H}$ is said to be a \emph{linear hypergraph}. In this paper we assume  that  hypergraphs are linear and $r$-uniform.

In a hypergraph $\mathcal{H}$,  two vertices $u$ and $v$ are {\em adjacent} if there is an edge $e$ of $\mathcal{H}$ such
that $\{u,v\}\subseteq e$. A vertex $v$ is said to be {\em incident} to an edge $e$ if $v\in e$.
A {\em walk} of hypergraph $\mathcal{H}$ is defined to be an alternating sequence of vertices and edges
$v_1e_1v_2e_2\cdots v_{\ell}e_{\ell}v_{\ell+1}$ satisfying that both $v_{i}$ and $v_{i+1}$ are
incident to $e_i$ for $1\leqslant i\leqslant\ell$. A walk is called a {\em path} if all vertices
and edges in the walk are distinct.
The walk is {\em  closed} if $v_{l+1}=v_1$. A closed walk is called a {\em  cycle} if all vertices and edges in the walk are distinct.
A
hypergraph $\mathcal{H}$ is called {\em connected} if for any vertices $u$, $v$, there is a path connecting
$u$ and $v$.
A hypergraph $\mathcal{H}$  is called \textit{acyclic} or a  \textit{hyperforest}  if it contains no cycle. A connected hyperforest  is called a \textit{hypertree}.

Let $\mathcal{H}=(V,E)$ be an   $r$-uniform hypergraph of order $n$ and size $m$.    \textit{A matching} of $\mathcal{H}$ is a set of pairwise
nonadjacent edges in $E$. A \textit{$k$-matching} is  a matching
consisting of $k$ edges. We denote by $m(\mathcal{H},k)$ the number of
$k$-matchings of $\mathcal{H}$. The  \textit{matching number} $\nu(\mathcal{H})$ of $\mathcal{H}$ is the maximum cardinality of a matching.

Recently, Zhang et. al ~\cite{Zhang_17} obtained the following result.

\begin{theorem}[\cite{Zhang_17}] \label{thm_Zhang_matchingpolyradius}
 $\lambda$ is a nonzero eigenvalue of a hypertree $\mathcal{H}$ with the corresponding
eigenvector $x$ having all elements nonzero if and only if it is a root of the polynomial
\begin{equation*}\label{e_matchenergy}
\varphi(\mathcal{H},x)=\sum\limits_{k=0}^{\nu(\mathcal{H})}(-1)^{k}m(\mathcal{H},k)x^{(\nu(\mathcal{H})-k)r}.
\end{equation*}
\end{theorem}
 We define the \textit{matching polynomial} of $\mathcal{H}$ in \cite{SuKLS} as
\begin{equation*}
\varphi(\mathcal{H},x)=\sum\limits_{k\geq0}(-1)^{k}m(\mathcal{H},k)x^{n-kr}.
\end{equation*}
This definition  seems more appropriate as it guarantees that  matching polynomials of hypergraphs of the same order have the same degree and the result in Theorem~\ref{thm_Zhang_matchingpolyradius} is still valid.

In \cite{LiShaoQi16},  some transformations on hypergraphs such as  ¡°moving edges¡±  and ¡°edge-releasing¡± were introduced and the first two spectral radii of hypertrees on $n$ vertices were characterized.  Yuan et. al~\cite{yuanshaoshan-supertrees-16} further determined the first eight uniform hypertrees on $n$ vertices with the largest spectral radii. Xiao et. al \cite{XiaoWanglu-supertree-degreesequ-17}  characterized the unique  uniform hypertree with the maximum spectral radius among all  uniform hypertrees with a given degree sequence. The first two largest spectral radii of uniform hypertrees with given diameter were characterized in \cite{XiaoWangDu-supertree-diam-18}. Recently Su  et. al \cite{SuKLS} determine the first    $\lfloor\frac{d}{2}\rfloor+1$  largest spectral radii of  $r$-uniform hypertrees with size $m$ and diameter $d$

In this paper, using the theory of matching polynomial of hypertrees introduced  in \cite{SuKLS},   we  determine the largest  spectral radius of hypertrees  with $m$ edges and given size of matching.
The structure of the remaining part of the paper is as follows: In Section 2, we give some basic definitions and results for tensor and spectra of hypergraphs. Section 3 is devoted to investigating various transformations on hypertrees with given size of matching.
In  last section,  the  largest spectral radius of hypertrees with $m$ edges and given size of matching is determined.

\section{Preliminaries}

Let $\mathcal{H}=(V, E)$ be an $r$-uniform hypergraph on $n$ vertices. A \textit{partial hypergraph}  $\mathcal{H}'=(V', E')$ of  $\mathcal{H}$ is a hypergraph with $V'\subseteq V$ and $E'\subseteq E$. A \textit{proper partial hypergraph} $\mathcal{H}'$ of  $\mathcal{H}$  is partial hypergraph of $\mathcal{H}$ with $\mathcal{H}'\neq\mathcal{H}$. For a vertex subset $S\subset V$,  let   $\mathcal{H}-S=(V'',E'')$ be the partial hypergraph of $\mathcal{H}$ satisfying that $V''=V\setminus S$, and for any $e\in E$, if $e\subseteq V''$, then $e\in E''$. When $S=\{v\}$,  $\mathcal{H}-S$ is simply written as $\mathcal{H}-v$. For an edge $e=\{v_1,\ldots,v_t\}\in E(\mathcal{H})$, let $\mathcal{H}\setminus e$ stand for the partial hypergraph of $\mathcal{H}$ obtained by   deletion of the edge $e$ from $\mathcal{H}$, i.e. $\mathcal{H}\setminus e=(V, E\setminus\{e\})$, and     $\mathcal{H}-V(e)$ stand for the partial hypergraph of $\mathcal{H}- \{v_1,\ldots,v_t\}$. A \textit{partial hypergraph} induced by an edge subset $F\subseteq E$ of $\mathcal{H}$  is a hypergraph $\mathcal{H}'=(V', F)$, where  $V'=\cup_{e\in F} e$.   For two $r$-uniform hypergraphs $\mathcal{G}$ and $\mathcal{H}$ with $V(\mathcal{G})\cap V(\mathcal{H})=\emptyset$,   we use $\mathcal{G}\dot{\cup}\mathcal{H}$ to denote the disjoint union of $\mathcal{G}$ and $\mathcal{H}$. Let $a$ be a positive integer, $a\mathcal{G}$ stand for the disjoint union of $a$  copies of $\mathcal{G}$.

  An edge $e$ of $\mathcal{H}$ is called a \textit{pendent edge} if $e$ contains
exactly $r-1$ core vertices. If $e$ is not a pendent edge,  it is  called a
\textit{non-pendent edge}.
An $r$-uniform hypergraph $\mathcal{H}$  is called a \textit{hyperstar}, denoted by $S_{m}^r$, if there is a disjoint partition of the vertex
set $V$ as  $V=\{v\}\cup V_1\cup\cdots\cup V_m$ such that  $|V_1|=\cdots=|V_m|=r-1$, and $E=\{\{v\}\cup V_i\ |  i=1,\ldots,m\}$, and $v$ is the \textit{center} of $S_{m}^r$.

For positive integers $r$ and $n$, a real
{\em tensor} $\mathcal{A}=(a_{i_1i_2\cdots i_r})$ of order $r$ and dimension $n$
refers to a multidimensional array (also called {\em hypermatrix}) with entries
$a_{i_1i_2\cdots i_r}$ such that $a_{i_1i_2\cdots i_r}\in\mathbb{R}$ for
all $i_1$, $i_2$, $\ldots$, $i_r\in[n]$, where $[n]=\{1,2,\ldots,n\}$.

The following product of tensors, defined by Shao \cite{Shao-tensorproduct},  is a generalization of
the matrix product.
Let  $\mathcal{A}$ and $\mathcal{B}$ be dimension $n$, order  $r\geqslant 2$
and order $k\geqslant 1$ tensors, respectively. Define the product
$\mathcal{AB}$ to be the tensor $\mathcal{C}$ of  dimension $n$ and order
$(r-1)(k-1)+1$ with entries as
\begin{eqnarray}\label{formu1}
c_{i\alpha_1\cdots\alpha_{r-1}}=\sum_{i_2,\ldots,i_r=1}^na_{ii_2\cdots i_r}
b_{i_2\alpha_1}\cdots b_{i_r\alpha_{r-1}},
\end{eqnarray}
where $i\in [n]$, $\alpha_1,\ldots,\alpha_{r-1}\in [n]^{k-1}$.

From the above definition, if $x=(x_1,x_2,\ldots,x_n)^{\mathrm{T}}\in \mathbb{C}^n$ is a complex
column vector of dimension $n$, then  by (\ref{formu1}) $\mathcal{A}x$ is a vector in $\mathbb{C}^n$ whose $i$th component
is given by
\begin{equation*}
\label{eq:Ax equation}
(\mathcal{A}x)_i=\sum_{i_2,\ldots,i_r=1}^na_{ii_2\cdots i_r}x_{i_2}\cdots x_{i_r},~~
\mbox{for each}\, \,i\in [n].
\end{equation*}

In 2005, Qi \cite{qi05} and Lim
\cite{Lim05} independently introduced the concepts of tensor eigenvalues and the spectra of tensors.

Let $\mathcal{A}$ be an order $r$ dimension $n$ tensor, $x=(x_1,x_2,\ldots,x_n)^{\mathrm{T}}\in\mathbb{C}^n$
 a column vector of dimension $n$. If there exists a number $\lambda\in\mathbb{C}$
and a nonzero vector $x\in\mathbb{C}^{n}$ such that
\begin{equation*}
\mathcal{A}x=\lambda x^{[r-1]},
\end{equation*}
where $x^{[r-1]}$ is a vector with $i$-th entry $x^{r-1}_i$, then $\lambda$ is called an {\em eigenvalue} of $\mathcal{A}$, $x$ is called
an {\em eigenvector} of $\mathcal{A}$ corresponding to the eigenvalue $\lambda$.
The {\em spectral
radius} of $\mathcal{A}$ is the maximum modulus of the eigenvalues of $\mathcal{A}$.

In 2012, Cooper and Dutle \cite{CoopDut12} defined the
adjacency tensors for $r$-uniform hypergraphs.

\begin{defn}
[\cite{CoopDut12}]
Let $\mathcal{H}=(V, E)$ be an $r$-uniform hypergraph on $n$ vertices. The adjacency
tensor of $\mathcal{H}$ is defined as the order $r$ and dimension $n$ tensor
$\mathcal{A}(\mathcal{H})=(a_{i_1i_2\cdots i_r})$, whose $(i_1i_2\cdots i_r)$-entry is
\[
a_{i_1i_2\cdots i_r}=\begin{cases}
\frac{1}{(r-1)!}, & \text{if}~\{i_1,i_2,\ldots,i_r\}\in E,\\
0, & \text{otherwise}.
\end{cases}
\]

\end{defn}

The \textit{spectral radius of hypergraph} $\mathcal{H}$ is defined as spectral radius  of its  adjacency tensor, denoted by $\rho(\mathcal{H})$.
In \cite{Fri} the weak irreducibility of nonnegative tensors was defined. It
was proved that an $r$-uniform hypergraph $\mathcal{H}$ is connected if and only if its adjacency
tensor $\mathcal{A}(\mathcal{H})$ is weakly irreducible (see \cite{Fri} and
\cite{YANGYANG-11}). Part of the Perron-Frobenius theorem for nonnegative tensors is
stated in the following for reference.
\begin{theorem}[\cite{QiLuo-2017}]
\label{thm:Perron-Frobenius}
Let $\mathcal{A}$ be a nonnegative tensor of order $r$ and dimension $n$, where $r, n\geq2$. Then   $\rho(\mathcal{A})$ is an eigenvalue of $\mathcal{A}$ with a nonnegative eigenvector corresponding to it.
If  $\mathcal{A}$ is weakly irreducible, then $\rho(\mathcal{A})$ is a positive
 eigenvalue of $\mathcal{A}$ with a positve eigenvector ${x}$. Furthermore,  $\rho(\mathcal{A})$ is the unique eigenvalue of  $\mathcal{A}$ with a positive eigenvector, and $x$ is the unique positive eigenvector associated with $\rho(\mathcal{A})$, up to a multiplicative  constant.
\end{theorem}
The unique positive eigenvector $x$ with $\sum_{i=1}^nx_i^r=1$  corresponding to $\rho(\mathcal{H})$ is called the \textit{principal eigenvector} of $\mathcal{H}$.

\begin{theorem}[\cite{YANGYANG-SIAM-10}]
\label{yang}
 Let $\mathcal{A}, \mathcal{B}$ be order $r$ and dimension $n$ nonnegative tensors, and $\mathcal{A}\not=\mathcal{B}$. If $\mathcal{B}\leq \mathcal{A}$ and $\mathcal{A}$ is weakly irreducible, then $\rho(\mathcal{A})>\rho(\mathcal{B})$.
\end{theorem}

The  following result can be obtained directly from Theorem \ref{yang} and will be often used in the sequel.

\begin{theorem}\label{YANGYANG-SIAM-10}
Suppose that $\mathcal{G}$  is a uniform hypergraph,   and $\mathcal{G}'$ is a partial hypergraph of $\mathcal{G}$. Then  $\rho(\mathcal{G}')\leq\rho(\mathcal{G})$. Furthermore, if in addition $\mathcal{G}$ is connected and $\mathcal{G}'$ is a proper partial hypergraph, we have  $\rho(\mathcal{G}')<\rho(\mathcal{G})$.
\end{theorem}

An operation of \textit{moving edges} on hypergraphs was introduced by Li et. al in \cite{LiShaoQi16}.
Let $\mathcal{H}=(V,E)$ be a hypergraph with $u\in V$ and $e_1,\cdots,e_k\in E$, such that $u\notin e_i$ for $i=1,\cdots,k$. Suppose that $v_i\in e_i$ and write $e_i'=(e_i\setminus \{v_i\})\cup \{u\}  \  (i=1,\cdots,k)$. Let $\mathcal{H}'=(V,E')$ be the hypergraph with $E'=(E\setminus\{e_1,\cdots,e_k\})\cup \{e_1',\cdots,e_k'\}$. Then we say that $\mathcal{H}'$ is obtained from $\mathcal{H}$ by \textit{moving edges} $(e_1,\cdots,e_k)$ from $(v_1,\cdots,v_k)$ to $u$.

 \begin{theorem}[\cite{LiShaoQi16}]\label{lem-edgemoving}
 Let $\mathcal{H}$ be a connected hypergraph, $\mathcal{H}'$ be the hypergraph obtained from $\mathcal{H}$ by moving edges $(e_1,\cdots,e_k)$  from $(v_1,\cdots,v_k)$ to $u$.  If $x$ is the principal eigenvector of $\mathcal{H}$ corresponding to $\rho(\mathcal{H})$, and suppose that $x_u\ge \max_{1\leq i\leq k}\{x_{v_i}\}$, then $\rho(\mathcal{H}')>\rho(\mathcal{H})$.
\end{theorem}

The following \textit{edge-releasing operation} on linear hypergraphs was given in \cite{LiShaoQi16}.

Let $\mathcal{H}$ be an $r$-uniform linear hypergraph,  $e$ be a non-pendent edge of $\mathcal{H}$ and $u\in e$.
  Let $e_1,e_2,\ldots,e_k$ be all  edges of $G$ adjacent to $e$ but not containing $u$, and suppose that $e_i\cap e=\{v_i\}$ for $i=1,\ldots,k$.    Let $\mathcal{H}'$ be the hypergraph obtained from $\mathcal{H}$ by moving edges $(e_{1},\cdots,e_k)$ from $(v_{1},\cdots,v_k)$ to $u$.  Then $\mathcal{H}'$ is said to be obtained by an \textit{edge-releasing operation} on $e$ at $u$.

  From the above definition we can see that if  $\mathcal{H}'$ and $\mathcal{H}''$ are the hypergraphs obtained from an $r$-uniform linear hypergraph $\mathcal{H}$ by an edge-releasing operation on some $e$ at $u$ and at $v$, respectively. Then $\mathcal{H}'$ and $\mathcal{H}''$ are isomorphic. So we simply say $\mathcal{H}'$ is obtained from $\mathcal{H}$ by an \textit{edge-releasing operation} on $e$.

Recall the ordering on hyperforests   introduced  in \cite{SuKLS}. Let $\mathcal{T}$ and  $\mathcal{T}'$ be hyperforests of $n$ vertices.  We call $\mathcal{T}'\preceq \mathcal{T}$  if $\varphi(\mathcal{T}',x)\geq \varphi(\mathcal{T},x)$ for every $x\geq \rho(\mathcal{T}')$; call $\mathcal{T}'\prec \mathcal{T}$  if  $\mathcal{T}'\preceq \mathcal{T}$ and the polynomial $\varphi(\mathcal{T}',x)-\varphi(\mathcal{T},x)$ does not vanish at the point $x=\rho(\mathcal{T}')$.  Note that $\mathcal{T}'\prec \mathcal{T}$ ($\mathcal{T}'\preceq \mathcal{T}$, resp.) implies $\rho(\mathcal{T}')< \rho(\mathcal{T})$ ($\rho(\mathcal{T}')\leq\rho(\mathcal{T})$, resp.).

We first give some useful results which were proposed in \cite{SuKLS}.

 \begin{lemma}[\cite{SuKLS}]\label{lem_subgraphmatchpoly}
If $\mathcal{T}$  is an  uniform hypertree,   and $\mathcal{T}'$ is a proper partial hypergraph of $\mathcal{T}$ with $V(\mathcal{T}')=V(\mathcal{T})$, then $\mathcal{T}'\prec\mathcal{T}$.
\end{lemma}

\begin{lemma}[\cite{SuKLS}]\label{lem-edgerelease}
Let $\mathcal{T}'$ be  an  $r$-uniform hypertree obtained by edge-releasing a non-pendent edge  of $\mathcal{T}$.  Then $\mathcal{T}'$ is a uniform hypertree and  $\mathcal{T}\prec \mathcal{T}'$.
\end{lemma}

 \begin{theorem}[\cite{SuKLS}]\label{thm_matchingpoly}
Let $\mathcal{G}$ and $\mathcal{H}$ be two $r$-uniform hypergraphs. Then the following  statements hold.
\begin{enumerate}
  \item $\varphi(\mathcal{G}\dot{\cup}\mathcal{H},x)=\varphi(\mathcal{G},x)\varphi(\mathcal{H},x)$.
\item $\varphi(\mathcal{G},x)=\varphi(\mathcal{G}\setminus e,x)-\varphi(\mathcal{G}-V(e),x)$ if $e$ is an edge of $\mathcal{G}$.
\item If $u\in V(\mathcal{G})$ and $I=\{i | e_i\in E_u\}$, for any $J\subseteq I$, we have
$$\varphi(\mathcal{G},x)=\varphi(\mathcal{G}\setminus \{e_i | i \in J\},x)-\sum_{i\in J}\varphi(\mathcal{G}-V(e_i),x)$$
and
$$\varphi(\mathcal{G},x)=x\varphi(\mathcal{G}-u,x)-\sum_{e\in E_u}\varphi(\mathcal{G}-V(e),x).$$
\end{enumerate}
\end{theorem}

A  hypergraph $\mathcal{H}=(V,E)$ is called a \textit{subtree hypergraph} if
  there is a tree $T$ with vertex set $V$ such that vertices in $e\in E$ induces a
subtree in $T$. It is not hard to see that hypertrees are subtree hypergraphs.

Let $\mathcal{H}=(V, E)$ be a hypergraph. A subfamily of  edges $F\subseteq E$  is an \textit{intersecting family} if every pair of  edges in $F$ has a non-empty intersection. A hypergraph
has the \textit{Helly property} if each intersecting family has a non-empty intersection.

\begin{lemma}[\cite{Bretto-hypergraph}]\label{lem-subtree-Helly}
Any subtree hypergraph   has the Helly property.
\end{lemma}

An $n$-dimensional vector $(a_1,\ldots,a_n)$ is said to be \textit{non-increasing} if $a_1\geq a_2\geq\ldots\geq a_n$.
Let $\pi=(a_1,\ldots,a_n)$ and $\pi'=(a_1',\ldots,a_n')$ be two non-increasing real vectors.  $\pi$ is said to be \textit{majorized} by  $\pi'$, which is denoted by $\pi\lhd\pi'$ (or $\pi'\triangleright\pi$),  if
\begin{align*}
&\sum_{i=1}^ka_i\leq \sum_{i=1}^ka_i',\,\, \textrm{for} \,\,   k=1,2,\ldots, n-1;\\
&\sum_{i=1}^na_i= \sum_{i=1}^na_i'.
\end{align*}

\begin{lemma}  \label{lem-majorization}
Let $\pi$  and $\pi'$ be two different non-increasing integral nonnegative vectors. If  $\pi\lhd\pi'$, then there exists a sequence of  non-increasing integral nonnegative vectors $\pi_1, \pi_2, \ldots, \pi_k$ such that $(\pi=)\pi_{k+1}\lhd \pi_{k}\lhd\cdots\lhd\pi_1\lhd\pi_{0}(=\pi')$, where  $\pi_i$  and $\pi_{i+1}$ differ only in two positions and  the differences are 1 for $0\leq i\leq k$.
\end{lemma}
\begin{proof}
Assume  $\pi=(a_1,a_2,\ldots,a_n)$  and $\pi'=(b_1,b_2,\ldots,b_n)$ are two different non-increasing integral nonnegative vectors with $\pi\lhd\pi'$.
Let $p$ be the smallest index such that $a_p>b_p$.
We claim such $p$ does exist. Otherwise,  $a_i\leq b_i$ for all $1\leq i\leq n$. Combining this   with the assumption that $\sum_{i=1}^na_i=\sum_{i=1}^nb_i$,  we obtain  that $a_i=b_i$ for $i=1,\ldots, n$, a contradiction to $\pi\neq\pi'$.
Furthermore,  there is an integer $i$ ($1\leq i<p$) such that $a_i< b_i$, and let   $q$ be the largest such index. Otherwise,
  $a_i\geq b_i$ for all $1\leq i<p$, then $\sum_{i=1}^{p-1}a_i\geq\sum_{i=1}^{p-1}b_i$ and $\sum_{i=1}^{p}a_i>\sum_{i=1}^{p}b_i$, a contradiction.  Then we have,
\begin{equation}\label{e-pq}
a_q< b_q,\quad a_p>b_p, \quad a_i= b_i, \quad \textrm{for} \,\, q<i<p.
\end{equation}
Now let $\pi_1=(b_1', b_2',\ldots, b_n')$ be a non-increasing integral nonnegative vector given by
\[b_i'=\left\{
  \begin{array}{ll}
    b_q-1, & \hbox{if $i=q$;} \\
    b_p+1, & \hbox{if $i=p$;} \\
    b_i, & \hbox{otherwise.}
  \end{array}
\right.
\]
It is easy to verify that $b_1'\geq b_2'\geq \cdots\geq b_n'$, and  $\pi_1\lhd \pi'$. Now we shall show that $\pi\lhd \pi_1$. First note that
$$\sum_{i=1}^ka_i\leq \sum_{i=1}^kb_i=\sum_{i=1}^kb_i',\,\,\textrm{for} \,\, 1\leq k\leq q-1\,\,\textrm{and} \,\, p\leq k\leq n.$$
For $q\leq k\leq p-1$, since $a_q\leq b_q-1$, we have
\begin{align*}
\sum_{i=1}^ka_i&=\sum_{i=1}^{q-1}a_i+ a_q+\sum_{i=q+1}^ka_i\\
&\leq \sum_{i=1}^{q-1}b_i+ a_q+\sum_{i=q+1}^ka_i\\
&\leq \sum_{i=1}^{q-1}b_i+ b_q-1+\sum_{i=q+1}^ka_i\\
&=\sum_{i=1}^kb_i'.
\end{align*}
So $\pi\lhd \pi_1$.
Furthermore, we have
\begin{align*}
\|\pi'-\pi\|_{1}-\|\pi_1-\pi\|_{1}&=\sum_{i=1}^n|b_i-a_i|-\sum_{i=1}^n|b_i'-a_i|\\
&=|b_q-a_q|+|b_p-a_p|- |b_q-1-a_q|-|b_p+1-a_p|\\
&=2.
\end{align*}

That is, $\|\pi_1-\pi\|_{1}=\|\pi'-\pi\|_{1}-2$.

Since $\pi\lhd \pi_1$, if $\pi\neq\pi_1$,  in a similar way we can determine a non-increasing  integral nonnegative  vector $\pi_2$ such that
\[
\pi\lhd\pi_2\lhd \pi_1,
\]
and $\|\pi_2-\pi\|_{1}=\|\pi_1-\pi\|_{1}-2$.
Since $\|\pi'-\pi\|$ is finite,
repeating the above process in a finite number of steps, we can obtain a sequence of vectors $\pi_{1}, \pi_{2},\ldots, \pi_{k}, \pi_{k+1}=\pi$ such that
$$\pi=\pi_{k+1}\lhd \pi_{k}\lhd\cdots\lhd\pi_1\lhd\pi_{0}=\pi',$$
and  $\pi_i$  and $\pi_{i+1}$ differ only in two positions, and the differences are 1 for $0\leq i\leq k$.
\end{proof}

For  positive integers $a,b$ and $c$, let $\mathcal{A}_{a,b}^c$ be  the set of all non-increasing integral nonnegative vectors of dimension $b$  as follows:
\begin{equation*}\label{e-Sabc}
\mathcal{A}_{a,b}^c=\{(x_1,\ldots,x_{b})\in \mathbb{Z}^b :  c\geq x_1\geq x_2\geq\cdots\geq x_b\geq 0, x_1+x_2+\cdots+x_b=a \}.
\end{equation*}

\section{Some transformations on hypertrees with a given size of matching}

Let $a\leq r$ be a positive integer, and   $R_a$ (see Fig.~1 (a)) be an $r$-uniform hypertree with  vertex set $V(R_a)=\{v_j^{(i)}:i=1,\ldots,a;j=1,\ldots,r\}\cup\{u_1,\ldots,u_{r-a}\}$ and edge set $E(\mathcal{T})=\{e,e_1,\ldots,e_a\}$,  where  $e=\{v_1^{(1)},\ldots,v_{1}^{(a)}, u_1,\ldots, u_{r-a}\}$  and $e_i=\{v_1^{(i)},\ldots,v_r^{(i)}\}$, $i=1,\ldots,a$. Let $\mathcal{T}$   be an  $r$-uniform hypertree with $v\in V(\mathcal{T})$. Let $a\leq r-1$ be an integer,  and $\mathcal{T}(v;a)$  (see Fig.~1 (b)) be the $r$-uniform hypertree  obtained from $R_a$ and $\mathcal{T}$  by identifying a core vertex of $e$ in $R_a$ and $v$ of $\mathcal{T}$. Let $a, b$ be two integers with $a,b\leq r-1$,  and $\mathcal{T}(v;a,b)$ (see Fig.~1 (c)) be the $r$-uniform hypertree  obtained from   $R_a$ and $\mathcal{T}(v; b)$  by identifying a core vertex of $e$  in $R_a$   and $v$ of $\mathcal{T}(v; b)$. We have the following results which will be used in our proof.

\begin{figure}[!hbpt]
\begin{center}
\includegraphics[scale=0.6]{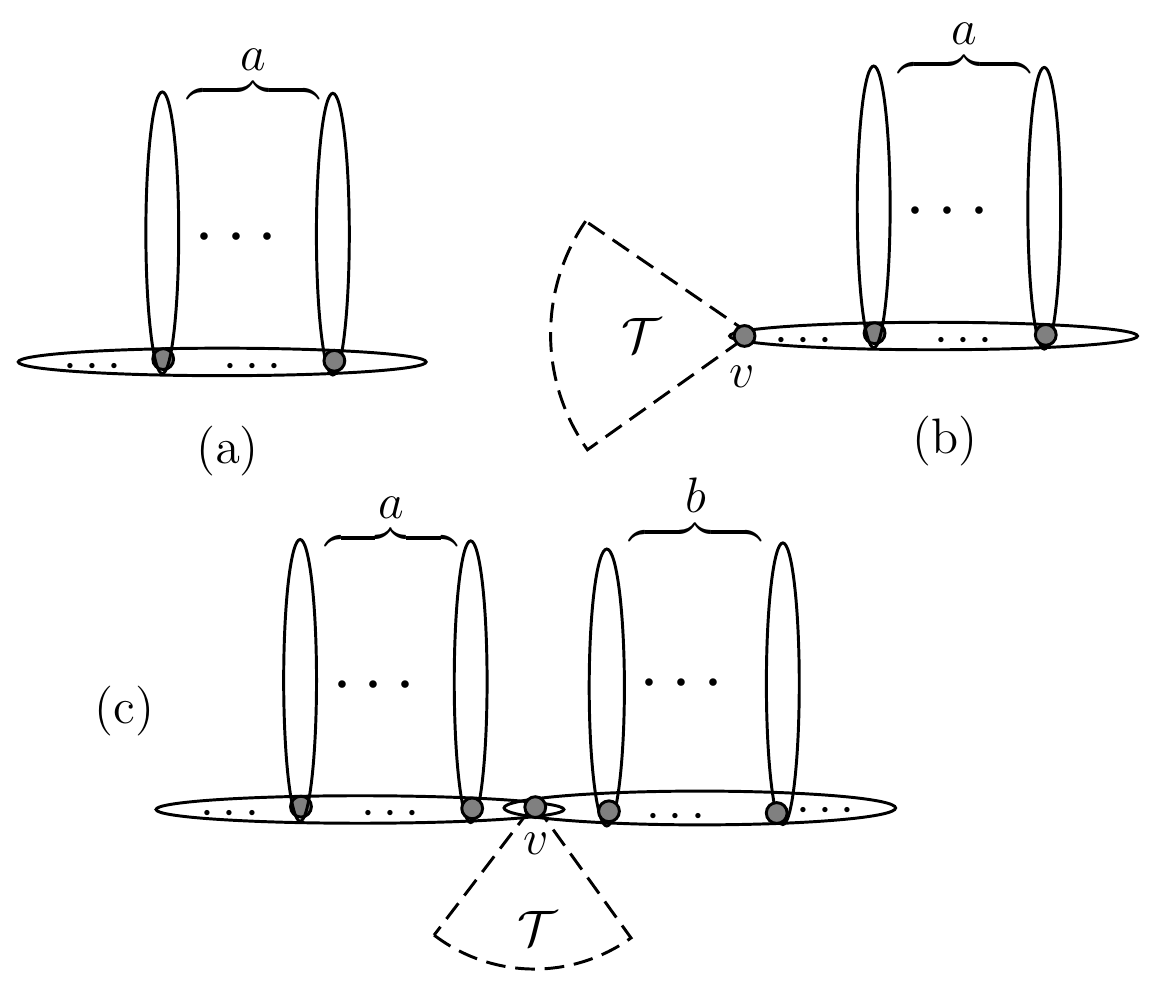}
\caption{Hypertrees  (a) $R_a$, (b) $\mathcal{T}(v;a)$, (c) $\mathcal{T}(v;a,b)$}.
\end{center}\label{figTab}
\end{figure}

\begin{lemma}\label{lem_RaTva}
Let $\mathcal{T}$ be an $r$-uniform hypertree,   and $R_a$  and $\mathcal{T}(v;a)$ be defined as above. Then
\begin{align}
&(a).  \hskip 6mm  \varphi(R_a,x)=x^{r-a}(x^r-1)^{a}-x^{a(r-1)}\nonumber \\
  &(b).  \hskip 6mm \varphi(\mathcal{T}(v;a),x)=x^{r-a-1}(x^r-1)^{a} \varphi(\mathcal{T},x)-x^{a(r-1)} \varphi(\mathcal{T}-v,x).\label{4}
\end{align}
\end{lemma}
 \begin{proof} (a).
Applying (b) of Lemma~\ref{thm_matchingpoly} to $R_a$ and the edge $e$ with $a$ intersection vertices, we have
\begin{align*}
  \varphi(R_a,x)=\varphi(R_a\setminus e,x)-\varphi(R_a-V(e),x)=x^{r-a}(x^r-1)^{a}-x^{a(r-1)}.
\end{align*}

(b).  Applying (b) of Lemma~\ref{thm_matchingpoly} to $\mathcal{T}(v;a)$ and the edge $e$ with $a+1$ intersection vertices, we have
\begin{align*}
  \varphi(\mathcal{T}(v;a),x)=\varphi(\mathcal{T}(v;a)\setminus e,x)-\varphi(\mathcal{T}(v;a)-V(e),x)=x^{r-a-1}(x^r-1)^{a} \varphi(\mathcal{T},x)-x^{a(r-1)} \varphi(\mathcal{T}-v,x).
\end{align*}
\end{proof}

\begin{lemma}\label{lem_Tab-Ta+1b-1}
Let   $\mathcal{T}$  be  an  $r$-uniform hypertree  with $v\in V(\mathcal{T})$, and $\mathcal{T}(v;a,b)$ be defined as above. If $r-2\geq a\geq b\geq1$, then $\mathcal{T}(v;a,b)\prec\mathcal{T}(v;a+1,b-1)$.
\end{lemma}
 \begin{proof}

 Applying (b) of Lemma~\ref{thm_matchingpoly} to $\mathcal{T}(v;a,b)$ and one  edge of $b$ pendent edges, we have
 \begin{align}
  \varphi(\mathcal{T}(v;a,b),x)=x^{r-1}\varphi(\mathcal{T}(v;a,b-1),x)-x^{r-1-b}(x^r-1)^{b-1}\varphi(\mathcal{T}(v;a),x).\label{5}
\end{align}

 Applying (b) of Lemma~\ref{thm_matchingpoly} to $\mathcal{T}(v;a+1,b-1)$ and one  edge of $a+1$ pendent edges, we have
\begin{align}
   \varphi(\mathcal{T}(v;a+1,b-1),x)=x^{r-1}\varphi(\mathcal{T}(v;a,b-1),x)-x^{r-2-a}(x^r-1)^{a}\varphi(\mathcal{T}(v;b-1),x).\label{6}
\end{align}

 By (\ref{4}),   (\ref{5}) and (\ref{6}),  we have
\begin{align*}\label{e-Tab-Ta+1b-1}
  &\varphi(\mathcal{T}(v;a,b),x)-\varphi(\mathcal{T}(v;a+1,b-1),x)\nonumber\\
  &=x^{r-2-a}(x^r-1)^{a}\varphi(\mathcal{T}(v;b-1),x)-x^{r-1-b}(x^r-1)^{b-1}\varphi(\mathcal{T}(v;a),x)\nonumber\\
&=x^{2r-a-b-2}(x^r-1)^{a+b-1}\varphi(\mathcal{T},x)-x^{b(r-1)-a-1}(x^r-1)^{a}\varphi(\mathcal{T}-v,x)\nonumber\\
&\quad-x^{2r-a-b-2}(x^r-1)^{a+b-1}\varphi(\mathcal{T},x)+x^{(a+1)(r-1)-b}(x^r-1)^{b-1}\varphi(\mathcal{T}-v,x)\nonumber\\
&=x^{b(r-1)-a-1}(x^r-1)^{b-1}\varphi(\mathcal{T}-v,x)(x^{(a+1-b)r}-(x^r-1)^{a+1-b}).
\end{align*}
Since  $\mathcal{T}-v$ is a proper partial hypergraph of $\mathcal{T}(v;a,b)$, we have $\rho(\mathcal{T}-v)<\rho(\mathcal{T}(v;a,b))$ by Theorem~\ref{YANGYANG-SIAM-10}. Thus if $x\geq \rho(\mathcal{T}(v;a,b))$, we have  $\varphi(\mathcal{T}-v,x)>0$ and then  $\varphi(\mathcal{T}(v;a,b))>\varphi(\mathcal{T}(v;a+1,b-1))$ considering $a\geq b$, namely $\mathcal{T}(v;a,b)\prec\mathcal{T}(v;a+1,b-1)$ when $a\geq b$.
\end{proof}

\begin{figure}[!hbpt]
\begin{center}
\includegraphics[scale=0.6]{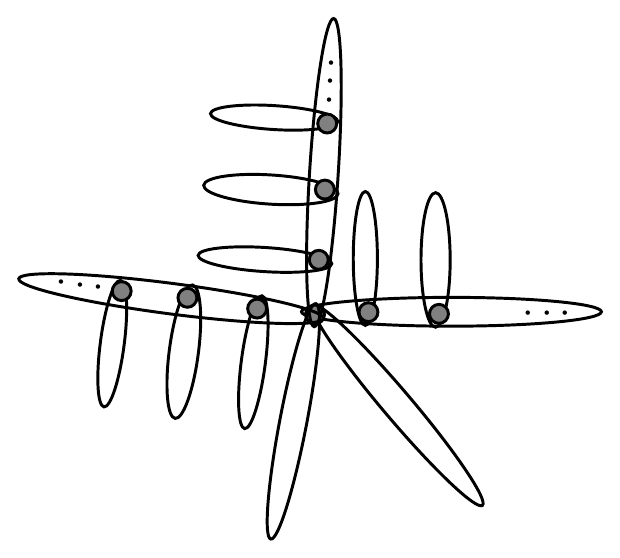}
\caption{Hypertrees   $S(3^{(2)},2,0^{(2)})$}
\end{center}\label{figTtildeabSab}
\end{figure}

\section{The first hypertree with  a given size of matching}

Let $S_t^r$ be a hyperstar  with edges $e_1,\ldots,e_t$ and center $v$. Denote by $S(a_1,a_2,\ldots,a_t)$  (see  Fig.~2 ) the $r$-uniform hypertree obtained from $S_t^r$ by attaching  $a_i$ disjoint pendent edges at distinct core vertices of $e_i$ for $i=1,\ldots,t$, and $v$ is called the \textit{center} of $S(a_1,a_2,\ldots,a_t)$.
 If $a_1=\cdots=a_s$ for some $s\leq t$,  $S(a_1,a_2,\ldots,a_t)$ is simply written as  $S(a_1^{(s)},a_{s+1},\ldots,a_t)$,  see Fig. 2.
Denote by $T_{m,k,r}$ the set of all $r$-uniform hypertrees with $m$ edges and a $k$-matching.
Let  $q,s,l$   be integers determined by $m,k,r$ as follows:
\begin{equation}\label{e-qsl}
\left\{
  \begin{array}{ll}
    k-1=(r-1)q+s, & \hbox{where $0\leq s< r-1$;} \\
    m=qr+s+1+l.
  \end{array}
\right.
\end{equation}

\noindent
 Let $A(m,k,r)$ be  the hypertree $S((r-1)^{(q)}, s,0^{(l)})$.

\begin{theorem}  \label{thm_maximum-radius-matchingnumber}
For any   $\mathcal{T}\in T_{m,k,r}$, we have
\begin{equation}\label{e-maxspecradi}
\rho(\mathcal{T})\leq \left(\frac{1}{1-\alpha_0}\right)^{1/r},
\end{equation}
 where $\alpha_0$ is  the maximum root of
\begin{equation}\label{10}
x^{r-1}\left(\frac{1}{1-x}-\frac{1}{x^s}-l\right)=q,
\end{equation}
and $q,s,l$ are defined as in \eqref{e-qsl}.
Further,   equality holds in \eqref{e-maxspecradi}  if and only if $\mathcal{T}=A(m,k,r)$.
\end{theorem}
\begin{proof}
Suppose that $\mathcal{T}_0\in T_{m,k,r}$  has the largest spectral radius. We shall show that $\mathcal{T}_0=A(m,k,r)$. We proceed the proof with the following claims.

\textbf{Claim 1.} Every edge  in  a $k$-matching of $\mathcal{T}_0$  is a  pendent edge.
  \begin{proof}
 Let $M=\{e_1\ldots,e_k\}$ be  a  $k$-matching of $\mathcal{T}_0$. If there exists an edge $e_i\in M$ ($1\leq i\leq k$) which is a non-pendent edge, then applying edge-releasing operation on $e_i$ we get a hypertree $\mathcal{T}'$ with $\rho(\mathcal{T}_0)<\rho(\mathcal{T}')$ by Lemma~\ref{lem-edgerelease}. It is easy to see that $\{e_1\ldots,e_k\}$ is still a  $k$-matching of $\mathcal{T}'$. So $\mathcal{T}'\in T_{m,k,r}$ and  $\rho(\mathcal{T}_0)<\rho(\mathcal{T}')$, contradicting with the maximality of $\mathcal{T}_0$.
\end{proof}

\textbf{Claim 2.} $\mathcal{T}_0=S(a_1,a_2,\ldots,a_{m-k},0)$, where $a_1,a_2,\ldots,a_{m-k}$ are nonnegative integers   with  $k=\sum_{i=1}^{m-k}a_i+1$.
  \begin{proof}
By Claim 1, we assume $M=\{e_1\ldots,e_k\}$ is a $k$-matching of $\mathcal{T}_0$ consisting of pendent edges.
 First we show that  $E\setminus M$  is an intersecting family in $\mathcal{T}_0$.
Suppose to the contrary that there exist two disjoint edges $e, f$ in $E\setminus M$. Since $\mathcal{T}_0$ is a hypertree,  there is a unique path $P$ connecting $e$ and $f$ in $\mathcal{T}_0$, say $P=v_1e_1v_2e_2\cdots,v_ae_av_{a+1}$, where $e_1=e$ and $e_a=f$. Let $\mathcal{T}_1$ and $\mathcal{T}_2$ denote the hypertrees obtained from $\mathcal{T}_0$ by moving edge $e_1$ from $v_2$ to $v_a$ and moving edge $e_a$ from $v_a$ to $v_2$, respectively. It is easy to see that both $\mathcal{T}_1$ and $\mathcal{T}_2$ are hypertrees. Since $e,f\notin M$, $M$ is still a $k$-matching of $\mathcal{T}_1$ and $\mathcal{T}_2$.
 However, by Theorem~\ref{lem-edgemoving}, we have $\rho(\mathcal{T}_0)<\max\{\rho(\mathcal{T}_1),\rho(\mathcal{T}_2)\}$, contradicting  to the maximality of $\mathcal{T}_0$. Thus $E\setminus M$  is an intersecting family in $\mathcal{T}_0$, and further by Lemma~\ref{lem-subtree-Helly}, we know that $E\setminus M$ has a common vertex, say $v$, and this is the  only common vertex since $\mathcal{T}_0$ is linear. Therefore the hypergraph induced by  $E\setminus M$ is a hyperstar, so $\mathcal{T}_0$ can be regarded as a hypertree obtained by attaching $k$ mutually disjoint pendent edges at hyperstar $S_{m-k}^r$.
By   the maximality of $\mathcal{T}_0$ and edge-moving operation in Theorem~\ref{lem-edgemoving}, it is not hard to see that there is exactly one  edge in $M$ attached at $v$. Thus there exist nonnegative integers $a_1,a_2,\ldots,a_{m-k}$ with   $k=\sum_{i=1}^{m-k}a_i+1$ such that
$\mathcal{T}_0=S(a_1,a_2,\ldots,a_{m-k},0)$.
\end{proof}

\textbf{Claim 3.} $\mathcal{T}_0=A(m,k,r)$.
  \begin{proof}

  Note that $S(a_1,a_2,\ldots,a_{m-k},0)=S_m^r=A(m,1,r)$ when $k=1$, and is $S(1,0^{(m-2)})=A(m,2,r)$ when $k=2$. So  we  assume $k\geq3$ in the following discussion.
By Claim 2,  we  assume that $\mathcal{T}_0=S(a_1,a_2,\ldots,a_{m-k},0)$, where $a_1,a_2,\ldots,a_{m-k}$ are nonnegative integers with $k=\sum_{i=1}^{m-k}a_i+1$. Note that  $\pi_0=(a_1,a_2,\ldots,a_{m-k})\in \mathcal{A}_{k-1,m-k}^{r-1}$ and
 $\pi_0\lhd(\underbrace{r-1,\ldots,r-1}_{q}, s,\underbrace{0,\ldots,0}_{l-1})$. By Lemma~\ref{lem-majorization}, there exist   a sequence of  vectors  $\pi_1, \pi_2, \ldots, \pi_p$ such that $\pi_{0}\lhd \pi_{1}\lhd\cdots\lhd\pi_{p}\lhd\pi_{p+1}$, where $\pi_{p+1}=(\underbrace{r-1,\ldots,r-1}_{q}, s,\underbrace{0,\ldots,0}_{l-1})$, $\pi_i$  and $\pi_{i+1}$ differ only in two positions and differ by one for $0\leq i\leq p$. It is easy to verify that $\pi_i\in \mathcal{A}_{k-1,m-k}^{r-1}$ for $i=0,\ldots, p+1$.
  Let $\pi_1=(b_1,b_2,\ldots,b_{m-k})$.  We assume that $\pi_1$ and $\pi_0$ differ by one in $i$th and $j$th components, with $b_i=a_i+1$ and $b_j=a_j-1$, $1\leq i<j\leq m-k$. Let   $\mathcal{H}=S(a_1,\ldots,a_{i-1},a_{i+1},\ldots,a_{j-1},a_{j+1},\ldots,a_{m-k},0)$ with $v$ as its center.  Then  $S(a_1,a_2,\ldots,a_{m-k},0)\cong \mathcal{H}(v; a_i, a_j)$ and $S(b_1,b_2,\ldots,b_{m-k},0)\cong \mathcal{H}(v;a_i+1, a_j-1)$.
By Lemma~\ref{lem_Tab-Ta+1b-1},  we have $ \mathcal{H}(v;a_i, a_j) \prec \mathcal{H}(v;a_i+1, a_j-1)$, i.e.
$$S(a_1,a_2,\ldots,a_{m-k},0) \prec S(b_1,b_2,\ldots,b_{m-k},0).$$

Repeatedly using  the above process, we get
 $S(a_1,a_2,\ldots,a_{m-k},0)\prec S(\underbrace{r-1,\ldots,r-1}_{q}, s,\underbrace{0,\ldots,0}_{l-1})=A(m,k,r)$.
By the maximality of $\mathcal{T}_0$, we have $\mathcal{T}_0=A(m,k,r)$.
\end{proof}

Applying (c) of Lemma~\ref{thm_matchingpoly} to $\mathcal{T}_0=A(m,k,r)$ on the center $v$, we obtain
\begin{align*}
&\varphi(\mathcal{T}_0,x)\\
&=x\varphi(\mathcal{T}_0-v,x)-\sum_{e\in E_v}\varphi(\mathcal{T}_0-V(e),x)\\
&=x^{(l+1)(r-1)-s+1}(x^r-1)^{q(r-1)+s}-qx^{(l+r)(r-1)-s}(x^r-1)^{(q-1)(r-1)+s}\\
&\,\,\,-x^{(l+s)(r-1)}(x^r-1)^{q(r-1)}-lx^{l(r-1)-s}(x^r-1)^{q(r-1)+s}\\
&=x^{(l+r)(r-1)-s}(x^r-1)^{(q-1)(r-1)+s}\left[x^r\left(\frac{x^r-1}{x^r}\right)^{r-1}-\left(\frac{x^r-1}{x^r}\right)^{r-1-s}-l\left(\frac{x^r-1}{x^r}\right)^{r-1}-q \right]
\end{align*}
Since the maximum root of $\varphi(\mathcal{T}_0, x)$ is $\rho(\mathcal{T}_0)$ and $\rho(\mathcal{T}_0)>1$,  $\rho(\mathcal{T}_0)$ is the maximum root of the following  equation
 \[x^r\left(\frac{x^r-1}{x^r}\right)^{r-1}-\left(\frac{x^r-1}{x^r}\right)^{r-1-s}-l\left(\frac{x^r-1}{x^r}\right)^{r-1}-q =0.\]
This  equation can be simplified to (\ref{10})  by variable replacement $x:=\frac{x^r-1}{x^r}$. The proof is completed.
\end{proof}

Theorem~\ref{thm_maximum-radius-matchingnumber} in this paper generalizes Theorem 3.3 in \cite{HouLi02} from trees to hypertrees.

When hypertree $\mathcal{T}$ has a perfect matching, we have the following result.
\begin{coro}
Let $\mathcal{T}$ be an $r$-uniform hypertree  with $m$ edges and a  perfect matching. Then
\begin{equation}\label{e-perfect-matching}
\rho(\mathcal{T})\leq \left(\frac{1}{1-\alpha_0}\right)^{1/r},
\end{equation}
where  $\alpha_0$ is the maximum root of
\[
rx^{r}=(m-1)(1-x).
\]
 Further, equality holds in \eqref{e-perfect-matching} if and only if $\mathcal{T}=A(\frac{rk-1}{r-1},k,r)$.
\end{coro}

\end{document}